\documentclass[preprint,12pt]{elsarticle}
\usepackage{amsmath}
\usepackage{amssymb}
\usepackage{tikz}
\usepackage{float}
\usetikzlibrary{arrows.meta}

\usepackage{babel}

\usepackage{hyperref}

\usepackage{epsfig}
\usepackage{graphicx}
\usepackage[cp1251]{inputenc}
\usepackage{amssymb}
\usepackage{amsthm}
\usepackage{cmap}

\usepackage{amsmath}

\biboptions{sort&compress}

\newcommand{\sysn}{\left\{\begin{array}{rcl}}
\newcommand{\sysk}{\end{array}\right.}

\newtheorem{theorem}{Theorem}[section]
\newtheorem{lemma}[theorem]{Lemma}

\theoremstyle{example}
\newtheorem{example}[theorem]{Example}

\theoremstyle{definition}

\newtheorem{corollary}[theorem]{Corollary}

\journal{...}

\begin{document}

\title{On properties close to Polish cofinality in spaces of Baire-one and continuous functions}

\author{Alexander V. Osipov}

\address{Krasovskii Institute of Mathematics and Mechanics, \\ Ural Federal
 University, Yekaterinburg, Russia}

\ead{OAB@list.ru}

\begin{abstract} In this paper we prove that $B_1(X)$ has a dense cofinally Polish subspace if and only if $X$ is a $\Delta_1$-space. Also we prove that $B_1(X,[0,1])$ has a dense cofinally Polish subspace if and only if $X$ is a weak $\lambda$-space. We also establish that $C_p(X,[0,1])$ is pseudocompact if and only if $C_p(X,[0,1])$ has a dense cofinally Polish subspace. This result provides answers to four open questions from (Tkachuk in RACSAM 115(68), 2021).
\end{abstract}


\maketitle 


\section{Introduction}

A Polish space is a separable, completely metrizable topological space, i.e., a space homeomorphic to a complete metric space that has a countable dense subset. Polish spaces are mostly studied today because they are the primary setting for descriptive set theory, including the study of Borel equivalence relations. Polish spaces are also a convenient setting for more advanced measure theory, in particular in probability theory.

A space $X$ is {\it cofinally Polish} if, for any continuous map $f : X \rightarrow M$ of $X$ onto a second
countable space $M$, there exist continuous onto maps $g : X \rightarrow P$ and $h : P \rightarrow M$ such that
$P$ is a Polish space and $h \circ g=f$.

In \cite{NiTkYe}, it was proved that a metrizable space is cofinally Polish if and only if it is Polish, and that for cosmic spaces, Polish cofinality coincides with domain representability.
Also in \cite{NiTkYe} it was proved that, for countable spaces, domain representability and subcompactness are equivalent to Polish cofinality.

In this paper, we investigate cofinally Polish property for the space $B_1(X)$ of Baire-one functions. We prove that $B_1(X)$ has a dense cofinally Polish subspace if and only if $X$ is a $\Delta_1$-space. Also we prove that $B_1(X,[0,1])$ has a dense cofinally Polish subspace if and only if $X$ is a weak $\lambda$-space.

We also establish that $C_p(X,[0,1])$ is pseudocompact if and only if $C_p(X,[0,1])$ has a dense cofinally Polish subspace. This result provides answers to four open questions in \cite{Tkachuk}.

\section{Definitions and notation}

 A space is {\it meager} (or {\it of the first Baire category}) if it
can be written as a countable union of closed sets with empty
interior. A topological space $X$ is {\it Baire} if the Baire
Category Theorem holds for $X$, i.e., the intersection of any
sequence of open dense subsets of $X$ is dense in $X$. Clearly, if
$X$ is a Baire space, then $X$ is not meager.

\medskip

Recall that $X\subseteq 2^{\omega}$ is

$\bullet$ a {\it  $\lambda$-set} if every countable subset of $X$ is $G_{\delta}$ (\S 40, \cite{Kur}).

$\bullet$ a {\it weak $\lambda$-set}  if for every pair of countable disjoint sets $Y,Z\subseteq X$ there exists a $G_{\delta}$-$F_{\sigma}$ set $H$ of $X$ such that $Y\subseteq H$ and $Z\subseteq X\setminus H$ (Def. 4.1 in \cite{RRS}).

The concepts of  $\lambda$-set and weak $\lambda$-set can be considered in arbitrary topological spaces.

A topological space $X$ is a {\it $\lambda$-space} if every countable subset of $X$ is $G_{\delta}$. A topological space $X$ is a {\it weak $\lambda$-space}  if for every pair of countable disjoint sets $Y,Z\subseteq X$ there exists a $G_{\delta}$-$F_{\sigma}$ set $H$ of $X$ such that $Y\subseteq H$ and $Z\subseteq X\setminus H$.

A family $\{A_{\alpha}: \alpha\in \kappa\}$ of subsets of a set $X$ is said to be {\it point-finite} if for every $x\in X$, $\{\alpha\in \kappa: x\in A_{\alpha}\}$ is finite.

A space $X$ has the {\it $\Delta_1$-property} \cite{KKL} if any disjoint sequence $\{A_n: n\in \omega\}$ of
countable subsets of $X$ has a point-finite open expansion, i.e., there exists a point-finite family $\{U_n: n\in \omega\}$ of open subsets of $X$ such that
$A_n\subset U_n$ for every $n\in \omega.$  The spaces with the $\Delta_1$-property are also called $\Delta_1$-spaces. It is clear that every $\lambda$-space is a $\Delta_1$-space.
If $X$ is a $\Delta_1$-space with a countable pseudocharachter then it is a $\lambda$-space (Theorem 2.19 in \cite{KKL}).

The Baire property have dual characterization through the presence (or
absence) of a winning strategy for players in the Banach-Mazur topological game \cite{Ox}.

The game $G(X)$ begins with the first $({\bf I})$ player; Player {\bf I} chooses a nonempty open set $V_0\subseteq X$.
Then the second ({\bf II}) player chooses a nonempty open set $V_1\subseteq V_0$.
At the {\it n}th step, player {\bf I} chooses a nonempty open set $V_{2n}\subseteq V_{2n-1}$, and then player {\bf II} chooses
a nonempty open set $V_{2n+1}\subseteq V_{2n}$.
Player {\bf I} wins the game $G(X)$ if $\bigcap\limits_{n\in \omega} V_n=\emptyset$. Otherwise, player {\bf II} wins.

 A topological space $X$ is a Baire space  if and only if player {\bf I} does not have a winning strategy in the game $G(X)$

A topological space $X$ is called a {\it Choquet space}  if player {\bf II} has a winning strategy in the game $G(X)$ \cite{Ox}.

It is clear that any Choquet space is a space with the Baire property.

The strong Choquet game $G^s(X)$ of a topological space $X$ \cite{Kechris} is a game between two players $\alpha$ and $\beta$ similar to $G(X)$. The player $\beta$ starts a play by
selection a pair $(x_0, U_0)$, where $U_0$ is an open subset of $X$ and $x_0\in U_0$. Then $\alpha$ must play a nonempty open subset $V_0$ of $U_0$ with $x_0\in U_0$.

In $n$th move, the player $\beta$ picks a pair $(x_n, U_n)$, where $U_n$ is an open subset of $V_{n-1}$, the previous move of $\alpha$-player, and $x_n\in U_n$. Then $\alpha$ answers by selecting an open set $V_n\subset U_n$ with $x_n\in V_n$. The player $\alpha$ wins the play $(U_i, V_i)_{i\in \omega}$, if $\bigcap_{n\in \omega} V_n\neq \emptyset$. Otherwise the player $\beta$ is said to have won the play.

We say that the player $\alpha$ has a winning strategy for the game $G^s(X)$ if there exists
a strategy $s$, such that $\alpha$ wins all plays provided that $\alpha$ acts according to
the strategy $s$. In this case, we say that $X$ is a {\it strong Choquet space}.

The class of strong Choquet spaces includes all countably base-compact spaces. A topological space $X$ is called {\it countably base-compact} if it has a base $\mathcal{B}$ of the topology such that for any countable centered subfamily $\mathcal{F}\subseteq \mathcal{B}$ the intersection $\bigcap_{F\in \mathcal{F}} \overline{F}$ is not empty. We recall that a family $\mathcal{F}$ of sets is called centered if any finite subfamily of $\mathcal{F}$ has nonempty intersection. The countable base-compactness is one of Amsterdam properties, discussed by Aarts and Lutzer in \cite{AL}.

\medskip
For topological spaces $X$ and $Y$ , we denote by $C_p(X,Y)$ the set $C(X,Y)$ of all continuous
functions from $X$ to $Y$, endowed with the topology of pointwise convergence. Denote by $C_p(X)=C_p(X,\mathbb{R})$.

Recall that a subset $A$ of a space $X$ is called a {\it zero-set} if $A=f^{-1}(0)$ for some $f\in C(X)$. A subset $A$ of a space $X$ is called a {\it cozero set} if $X\setminus A$ is a zero-set.

 A function $f : X \rightarrow Y$
from a topological space $X$ to a topological space $Y$ is called a Baire-one function  if $f$ is the pointwise limit of a sequence $\{f_n : n\in \omega\}\subset C(X, Y )$. Let $B_1(X, Y )$ denote the set of all Baire-one functions from a topological space $X$ to a topological space $Y$,
endowed with the topology of pointwise convergence. Denote by $B_1(X)=B_1(X,\mathbb{R})$.

\medskip

\medskip

A set $Y\subset X$ is called

$\bullet$ {\it $B_1$-embedded} in $X$ if for any $f\in B_1(Y)$ there exists $\tilde{f}\in B_1(X)$ such that $\tilde{f}|Y=f$;

$\bullet$ {\it $B^*_1$-embedded} in $X$ if for any bounded  $f\in B_1(Y)$ there exists $\tilde{f}\in B_1(X)$ such that $\tilde{f}|Y=f$.

\medskip

Let us recall (\cite{Arh1}, \S 6.1)
that a topological space is Moscow if the closure of every open subset is a union of $G_{\delta}$-sets. By (\cite{Arh1}, p. 347)
the Tychonoff product of first-countable space is Moscow and every dense subspace of a Moscow space is
Moscow.

\begin{theorem}\label{th11}(Theorem 3.8 in \cite{Os2}) For a Tychonoff space $X$ the following assertions are equivalent.

(1)   $B_1(X,[0,1])$ is a Choquet space.

(2)   Every countable subset of $X$ is $B^*_1$-embedded.

(3) $B_1(X,[0,1])$ is pseudocompact.

(4)  $X$ is a weak $\lambda$-space.

\end{theorem}

A family $\mathcal{F}\subseteq Y^{X}$ of functions from a set $X$ to a set $Y$ is called {\it $\omega$-full in $Y^X$} if each $f:Z\rightarrow Y$ defined on a countable subset $Z\subseteq X$ has an extension $\overline{f}\in \mathcal{F}.$

\begin{theorem}\label{th1} (Theorems 3.6 and 3.7 in \cite{Os2}) For any space $X$ the following assertions are equivalent:

(1)   $B_1(X)$ is a Choquet space;

(2)   Every countable subset of $X$ is $B_1$-embedded;

(3)  $X$ is a $\Delta_1$-space;

(4)  $B_1(X)$ is pseudocomplete;

(5) $B_1(X)$ is a $G_{\delta}$-dense subset in $\mathbb{R}^X$;

(6) $B_1(X)$ is $\omega$-full in $\mathbb{R}^X$.

\end{theorem}

The rest of our notation is standard and follows the book \cite{Eng}. The books  \cite{Tkachuk2} and \cite{Tk14} contain all
necessary facts and notions of $C_p$-theory.

\section{Properties of spaces of Baire-one functions}

We prove that the pseudocompleteness of the space $B_1(X)$ is equivalent to the following additional conditions.

\begin{theorem}\label{th2} For any space $X$ the following assertions are equivalent:

(1)  $B_1(X)$ is pseudocomplete;

(2) $B_1(X)$ is cofinally Polish;

(3) $B_1(X)$ has a dense cofinally Polish subspace;

(4) $B_1(X)$ is a strong Choquet space;

(5) $B_1(X)$ is countably base-compact.

\end{theorem}

\begin{proof} By Theorem 2.21 in \cite{BG} and Theorem 4.5 in \cite{Osip1}, it holds that (1) is equivalent to (4). $(2)\Rightarrow (3)$ is trivial.

$(3)\Rightarrow (1)$. This implication is proved analogously to the proof of Theorem 3.4 in \cite{Tkachuk} for $C_p$-spaces. Assume that $Z$ is a dense cofinally Polish subspace of the space $B_1(X)$. Take any countable set $A\subset X$ and observe that $Z$ is  dense in $\mathbb{R}^X$ so we can apply Theorem 3.3 in \cite{Tkachuk} to see that $p_A(Z)$ has a dense Polish subspace $P$, where $p_A: \mathbb{R}^X\rightarrow \mathbb{R}^A$ is the projection.

Observe that $p_A(Z)$ is dense $B_1(A|X)$ and assume that $B_1(A|X)\neq \mathbb{R}^A$. Pick any function $f\in \mathbb{R}^A\setminus B_1(A|X)$ and note that $B_1(A|X)$ does not meet its coset $B_1(A|X)+f$. As a consequence, $p_A(Z)$ and $p_A(Z)+f$ are two dense disjoint \v{C}ech-complete subspaces of $\mathbb{R}^A$
which is impossible by Problem 264 of \cite{Tkachuk2}. This contradiction shows that  $B_1(A|X)=\mathbb{R}^A$
for any countable set $A\subset X$ or, in other words, that every countable subset of $X$ is $B_1$-embedded. By Theorem \ref{th1}, $B_1(X)$ is pseudocomplete.

$(1)\Rightarrow (2)$. Given any continuous onto map $\varphi: B_1(X)\rightarrow M$ for some second countable space $M$, there exists a countable set $A\subset X$ and a continuous map $\psi: p_A(B_1(X))\rightarrow M$ such that $\psi\circ p_A=\varphi$ (see \cite{Tkachuk2}, Problem 299). All countable subsets of $X$ are $B_1$-embedded by Theorem \ref{th1} so $P=p_A(B_1(X))=\mathbb{R}^A$ is a Polish space which witnesses that $B_1(X)$ is cofinally Polish.

Since $B_1(X)$ is dense in $\mathbb{R}^X$, $B_1(X)$ is strong Choquet if and only if it is countably base-compact (see Theorem 2.21 in \cite{BG}).
\end{proof}

Note that if $X$ has a dense pseudocomplete subspace (in particular, if $X$ has a dense
\v{C}ech-complete subspace) then $X$ is a pseudocomplete space (\cite{Tkachuk2}, p. 47).

\begin{corollary} {\it $B_1(X)$ has a dense cofinally Polish subspace if and only if $X$ is a $\Delta_1$-space.}
\end{corollary}

\begin{corollary} {\it If $B_1(X)$ is cofinally Polish, then $B_1(X, [0,1])$ is also cofinally Polish.}
\end{corollary}

\begin{proof} This is because $B_1(X)$ must be pseudocomplete by Theorem \ref{th2} and hence $B_1(X, [0,1])$ is pseudocomplete (\cite{Tkachuk2}, p. 47). By Theorem 3.8 in \cite{Os2}, $B_1(X, [0,1])$ is pseudocompact. Since every pseudocompact space is
cofinally Polish by Proposition 3.18 of \cite{NiTkYe}, the space $B_1(X,[0,1]$ must be cofinally Polish.
\end{proof}

\begin{corollary}{\it If $X$ is  a space with a countable pseudocharachter then $B_1(X)$ has a dense cofinally Polish subspace if and only if $X$ is a $\lambda$-space.}
\end{corollary}

\begin{theorem}\label{th3} For any space $X$ the following assertions are equivalent:

(1)  $B_1(X,[0,1])$ is pseudocomplete;

(2) $B_1(X,[0,1])$ is cofinally Polish;

(3) $B_1(X,[0,1])$ has a dense cofinally Polish subspace;

(4) $B_1(X,[0,1])$ is pseudocompact;

(5) $B_1(X,[0,1])$ is a strong Choquet space;

(6) $B_1(X,[0,1])$ is countably base-compact.

\end{theorem}

\begin{proof}  $(3)\Rightarrow (4)$.  Since $(2)\Rightarrow (3)$ is trivial, assume that $B_1(X,[0,1])$ has a dense
cofinally Polish subspace $Z$.  Fix any countable set $A\subset X$ and
observe that the set $Z$ is dense in the space $[0,1]^X$ so we can apply Theorem 3.3 in \cite{Tkachuk} to conclude that $p_A(Z)$ has a dense Polish subspace; here $p_A: [0,1]^X \rightarrow [0,1]^A$ is the projection. Since $p_A(Z)\subset B_1(A|X,[0,1])$, we can choose a dense Polish set $E$ in $B_1(A|X,[0,1])$.
 Let $f\in [0,1]^A\setminus B_1(A|X,[0,1])$. Then $f+E$ is a dense Polish subspace of $[0,1]^A$. Since
$[0,1]^A$ is Baire, $(f+E)\cap E\neq \emptyset$, it is a contradiction. Thus, $[0,1]^A=B_1(A|X,[0,1])$ and, hence, $A$ is $B^*_1$-embedded.

Since every countable subset of $X$ is $B^*_1$-embedded, $B_1(X,[0,1])$ is pseudocompact by Theorem \ref{th11}.

$(4)\Rightarrow(2)$. By Theorem 3.18   in \cite{NiTkYe}.

$(1)\Leftrightarrow(4)$. By Theorem 3.8 in \cite{Os2}.

By Proposition 2.19 in \cite{BG}, every $G_{\delta}$-dense subspace $X$ of a countably base-compact Moscow space $Y$ is countably base-compact.

\end{proof}

\begin{corollary}\label{cor10} {\it $B_1(X,[0,1])$ has a dense cofinally Polish subspace if and only if $X$ is a weak $\lambda$-space.}
\end{corollary}

\section{Properties of spaces of continuous functions}

Recall that $S$ is {\it $C$-embedded} in $X$ if every function in $C(S)$ can be extended to a function in $C(X)$.

\begin{theorem}\label{th5} For any space $X$, the following statements are equivalent:

(1) $C_p(X)$ is pseudocomplete;

(2) $C_p(X)$ is a strong Choquet space;

(3) $C_p(X)$ is countably base-compact;

(4) $C_p(X)$ is cofinally Polish;

(5) $C_p(X)$ has a dense cofinally Polish subspace;

(6) Every countable subset of $X$ is $C$-embedded and closed;

(7) $C_p(X)$ is $G_{\delta}$-dense in $\mathbb{R}^X$;

(8) $C_p(X)$ is $\omega$-full in $\mathbb{R}^X$;

(9)  $C_p(X)$ is a Choquet space.

\end{theorem}

\begin{proof}
$(1)\Leftrightarrow(6)$. By Theorem 2 in \cite{Pyt}.

$(1)\Leftrightarrow(4)\Leftrightarrow(5)$ By Theorem 3.4 in \cite{Tkachuk}.

Note that if every countable subset $S$ of $X$ is closed, $S$ is discrete subspace of $X$. Then every countable subset of $X$ is $C$-embedded and closed if and only if $C_p(X)$ is $\omega$-full in $R^X$.
By Theorem 2.21 in \cite{BG}, $(6)\Leftrightarrow(2)\Leftrightarrow(3)\Leftrightarrow(7)\Leftrightarrow(8)\Leftrightarrow(9)$.

\end{proof}

A subset $S$ of a space $X$ is said to be {\it $C^*$-embedded}  in $X$ \cite{GJ} if every continuous real-valued bounded
function on $S$ can be extended to a continuous function on $X$. Urysohn's extension theorem (\cite{GJ}, 1.17) says
that a subset $S$ of a space $X$ is $C^*$-embedded in $X$ if and only if for any disjoint zero-sets $A$ and $B$ in $S$,
there are disjoint zero-sets $Z_A$ and $Z_B$ in $X$ such that $A\subseteq Z_A$ and $B\subseteq Z_B$.

According to Tkachuk \cite{Tk86}, a space $X$ is said to be {\it $b$-discrete} if every countable subset of $X$ is closed
(equivalently, closed and discrete) and $C^*$-embedded in $X$.


Recall that a mapping $\varphi: K\rightarrow M$ is called {\it almost open} if for every nonempty open set $V\subseteq K$,
$Int(\varphi(V))=\emptyset$. The class of almost open mappings includes both open mappings and irreducible
mappings of compacta. Note that if $\varphi_{\alpha}: K_{\alpha}\rightarrow M_{\alpha}$ is an almost open surjective mapping for each
$\alpha\in A$, then the product of mappings $\prod_{\alpha\in A} \varphi_{\alpha}: \prod_{\alpha\in A} K_{\alpha}\rightarrow \prod_{\alpha\in A} M_{\alpha}$ is also an almost open
surjective mapping.

\begin{lemma}\label{lem1}(Lemma 3.1 in \cite{Os2}) Let $\psi: P\rightarrow L$ be a surjective continuous almost open mapping. If the set $E\subseteq P$
is dense in $P$ and is a Choquet space, then the set $\psi(E)$ is a Choquet space.
\end{lemma}

The following theorem answers four open questions (Questions 4.1-4.4 in \cite{Tkachuk} or Problems 87-90 in \cite{Tk22}).

\begin{theorem}\label{th6} For any space $X$, the following statements are equivalent:

(1) $C_p(X,[0,1])$ is pseudocomplete;

(2) $C_p(X,[0,1])$ is pseudocompact;

(3) $C_p(X,[0,1])$ is a strong Choquet space;

(4) $C_p(X,[0,1])$ is countably base-compact;

(5) $C_p(X,[0,1])$ is cofinally Polish;

(6) $C_p(X,[0,1])$ has a dense cofinally Polish subspace;

(7) $X$ is $b$-discrete.
\end{theorem}

\begin{proof} $(1)\Rightarrow(7)$. Let $C_p(X,[0,1])$ be a pseudocomplete space. Then it is Choquet.  Consider a continuous
almost open mapping $\varphi: [0,1]\rightarrow \mathbb{S}^1$ of the interval $[0,1]$ onto the circle $\mathbb{S}^1$, identifying the endpoints
$0$ and $1$ ($\mathbb{S}^1=\mathbb{R}/\mathbb{Z}$ is considered as the quotient group of the group $\mathbb{R}$ by the subgroup of integers $\mathbb{Z}$).
The mapping $\varphi^X: [0,1]^X \rightarrow (\mathbb{S}^1)^X$ is a surjective continuous almost open mapping. Since
$C_p(X,[0,1])$ is dense in $[0,1]^X$ and is a Choquet space, by Lemma \ref{lem1} $\varphi^X(C_p(X,[0,1]))$ is a Choquet
space. But $\varphi^X(C_p(X,[0,1])=\{\varphi\circ f: f\in C_p(X,[0,1])\}\subseteq C_p(X,\mathbb{S}^1)$, i.e., $C_p(X,\mathbb{S}^1)$  contains
a dense Choquet space $\varphi^X(C_p(X,[0,1]))$ and, consequently, is itself a Choquet space.
Since $\mathbb{S}^1$ is a topological group, so is $C_p(X, \mathbb{S}^1)$ (under pointwise operations).
In (\cite{BH},Theorem 3) it is proved that a topological group $G$ is a Choquet space if and only if the
Raikov completion $\widetilde{G}$ is a Choquet space and $G$ is a $G_{\delta}$-dense subset of the space $\widetilde{G}$. It follows that
the space $C_p(X, \mathbb{S}^1)$ is $G_{\delta}$-dense in the space $(\mathbb{S}^1)^X$.
In  (\cite{BG}, Proposition 2.16) it is proved that for arbitrary topological spaces $X$, $Y$ and $\mathcal{F}\subseteq Y^X$,
where $Y$ is of countable pseudocharacter, $G_{\delta}$-density of $\mathcal{F}$ in the space $Y^X$ is equivalent to $\omega$-fullness
of $\mathcal{F}$ in $Y^X$. It follows that the space $C_p(X, \mathbb{S}^1)$ is $\omega$-full in the space $(\mathbb{S}^1)^X$.
Let $A$ and $B$ be two disjoint countable subsets of the space $X$ and $a$ and $b$ be two distinct points
in the space $\mathbb{S}^1$. By virtue of the $\omega$-fullness of the space $C_p(X, \mathbb{S}^1)$, a mapping $f : A\cup B \rightarrow \mathbb{S}^1$ such
that $f(A)\subseteq \{a\}$ and $f(B)\subseteq \{b\}$ extends to a mapping 	$\widetilde{f}\in C_p(X, \mathbb{S}^1)$, i.e., 	 $\widetilde{f}|(A\cup B) = f$.
It remains to note that $\widetilde{f}^{-1}(a)$ and $\widetilde{f}^{-1}(b)$	are two disjoint zero-sets such that $A\subseteq \widetilde{f}^{-1}(a)$ and
$B\subseteq \widetilde{f}^{-1}(b)$. Thus, every countable subset of $X$ is $C^*$-embedded and closed.

$(7)\Rightarrow(2)$. By Lemma 3.4 in \cite{Tk86}, $C_p(X,[0,1])$ is pseudocompact.

 Since any Tychonoff power $[0,1]^X$ of the segment $[0,1]$ is a Moscow space, Proposition 2.19 in \cite{BG} implies that each $G_{\delta}$-dense subspace $C_p(X,[0,1]$ of $[0,1]^X$ is countably base-compact. Thus, $(7)\Rightarrow(4)$ holds.

$(4)\Rightarrow(3)$. It is trivial.

Since any strong Choquet space is Choquet, $(3)\Rightarrow(1)$ by Theorem 4.1 in \cite{Os2}.

$(2)\Rightarrow(5)$. Every pseudocompact space is cofinally Polish (Proposition 3.18 in \cite{NiTkYe}).

$(5)\Rightarrow(6)$. It is trivial.

$(6)\Rightarrow(7)$. Assume that $C_p(X,[0,1])$ has a dense
cofinally Polish subspace $Z$.  Fix any countable set $A\subset X$ and
observe that the set $Z$ is dense in the space $[0,1]^X$ so we can apply Theorem 3.3 in \cite{Tkachuk} to conclude that $p_A(Z)$ has a dense Polish subspace; here $p_A: [0,1]^X \rightarrow [0,1]^A$ is the projection. Since $p_A(Z)\subset C_p(A|X,[0,1])$, we can choose a dense Polish set $E$ in $C_p(A|X,[0,1])$.
 Let $f\in [0,1]^A\setminus C_p(A|X,[0,1])$. Then $f+E$ is a dense Polish subspace of $[0,1]^A$. Since
$[0,1]^A$ is Baire, $(f+E)\cap E\neq \emptyset$, it is a contradiction. Thus, $[0,1]^A=C_p(A|X,[0,1])$ and hence $A$ is $C^*$-embedded.

\end{proof}

\section{Projectively Menger property}

A space $X$ is said to be {\it Menger} \cite{Hur}  if for every sequence $\{\mathcal{U}_n: n\in \mathbb{N}\}$ of open
covers of $X$, there are finite subfamilies $\mathcal{V}_n\subset \mathcal{U}_n$ such that $\bigcup \{\mathcal{V}_n: n\in \mathbb{N}\}$ is a cover of $X$.

A space $X$ is {\it projectively Menger}  if every second
countable continuous image of $X$ is Menger.

Every Menger space is projectively Menger. It is known (\cite{BCM},  Proposition 8 (1)) that a space is Menger if
and only if it is Lindel\"{o}f and projectively Menger.


Among several characterizations of a projectively Menger space, the following one is useful.

\begin{theorem} (\cite{BCM}, Theorem 6). A space $X$ is projectively Menger if and only if for every sequence $\{\mathcal{U}_n :
n\in \mathbb{N}\}$ of countable covers of $X$ consisting of cozero-sets in $X$, there are finite subfamilies $\mathcal{V}_n\subseteq \mathcal{U}_n$ such
that $\bigcup\{\mathcal{V}_n: n\in\mathbb{N}\}$ is a cover of $X$.
\end{theorem}

In (\cite{Sakai}, Theorem 2.4) it is proved that a space $X$ is $b$-discrete if and only if $C_p(X,[0,1])$ is projectively Menger.

Similarly, we prove the criterion for projective Menger property for $B_1(X,[0,1])$.

\begin{theorem} A space $X$ is a weak $\lambda$-space if and only if $B_1(X,[0,1])$ is projectively Menger.
\end{theorem}

\begin{proof} If $X$ is a weak $\lambda$-space, $B_1(X,[0,1])$ is pseudocompact (Theorem \ref{th3}) and hence it is projectively Menger.

Conversely, assume that $B_1(X,[0,1])$ is projectively Menger.  Let $A$ and $B$ be disjoint countable subsets in $X$. Let $B=\{b_n: n\in \mathbb{N}\}$, and let
$B_n =\{b_1,..., b_n\}$. For each $n$, $m\in \mathbb{N}$, we put

$Z_{n,m}=\{f\in B_1(X,[0,1]): f (A) =\{0\}$ and $f(B_m)\subset[\frac{1}{2^n}, 1]\}$.

Since $A$ and $B_m$ are countable, each $Z_{n,m}$ is a zero-set in $B_1(X,[0,1])$.

Assume that $\bigcap \{Z_{n,m} : m\in \mathbb{N}\}=\emptyset$ for
all $n\in \mathbb{N}$. Using the projective Menger property of $B_1(X,[0,1])$, by Theorem 6 in \cite{BCM}, we can take some $\varphi\in \mathbb{N}^\mathbb{N}$
such that $\bigcap \{Z_{n,\varphi(n)}: n\in \mathbb{N}\}=\emptyset.$ For each $n\in \mathbb{N}$, take any $g_n\in B_1(X,[0,1])$ satisfying $g_n(A) =\{0\}$ and
$g_n(B_{\varphi(n)}) =\{1\}$.

Let $g =\sum\limits_{j=1}^{\infty} 2^{-j} g_j$. Then, $g\in B_1(X,[0,1])$ and $g(A)=\{0\}$. Fix any $n\in \mathbb{N}$ and $1\leq k\leq \varphi(n)$.
Then we have

$$g(b_k)=\sum\limits_{j=1}^{\infty} 2^{-j} g_j(b_k)\geq 2^{-n}g_n(b_k)=2^{-n}.$$
Hence, $g\in \bigcap \{Z_{n,\varphi(n)}: n\in \mathbb{N}\}$. This is a contradiction. Thus, there is some $n\in \mathbb{N}$ such that $\bigcap \{Z_{n,m} : m\in \mathbb{N}\}\neq\emptyset$.
Let $h\in \bigcap \{Z_{n,m} : m\in \mathbb{N}\}$, then $A\subseteq h^{-1}(0)$ and $B\subseteq h^{-1}([\frac{1}{2^n}, 1])$. Thus, $A$ and $B$ are $G_{\delta}$ separated.

\end{proof}

\section{Exampes}

\begin{example}\label{ex1} If $X$ is the Shakhmatov's example of an infinite pseudocompact $b$-discrete space \cite{Sh}. Then $C_p(X,[0,1])$ is pseudocompact and hence cofinally Polish by Theorem \ref{th6}. Since any countable subset $S$ of $X$ is closed, $S$ is a discrete subspace. Let $S=\{s_i: i\in \mathbb{N}\}$. Then $f:S\rightarrow \mathbb{R}$ such that $f(s_i)=i$ is not extended on $X$ because $X$ is pseudocompact. Thus $S$ is not $C$-embedded  and therefore $C_p(X)$ is not cofinally Polish by Theorem \ref{th5}.
\end{example}

Examples of spaces with similar properties can be found in \cite{Is}, \cite{Arh2}.

\begin{example} Let $X$ be a $\lambda$-set. Since any $\lambda$-set is a $\Delta_1$-space, $B_1(X)$ is cofinally Polish by Theorem \ref{th2}.
Since $X$ is a separable metrizable (non-discrete) space, $C_p(X)$ is not Baire and therefore is not cofinally Polish by Theorem \ref{th5}.
\end{example}

In (\cite{JSS}, Theorem 4.1) it is proved under $ZFC$ that for each infinite car-
dinal $\kappa$ the Cantor cube $D^{2^{\kappa}}$
contains a dense submaximal subspace $X$ with $|X| =\Delta(X) =\kappa$. Then in the case $\kappa=\omega$ there is a countable dense submaximal subspace $X$ in the Cantor cube $D^{\mathfrak{c}}$.
Note that any countable dense subspace of $D^{\mathfrak{c}}$ is homeomorphic to
a dense subspace of $[0,1]^{\mathfrak{c}}$ (Fact 4.2 in \cite{JSS}).

\begin{example} Consider a space $X\subseteq [0,1]^{\mathfrak{c}}$ (Theorem 4.5 in \cite{OS1}) such that:

(a)  $X=S\cup P$ where $S$ and $P$  are disjoint dense subsets of $[0,1]^{\mathfrak{c}}$;

(b) $S$ is a countable submaximal space;

(c) $P$ is the Shakhmatov's space, i.e. $P$ is a pseudocompact $b$-discrete space (see Example \ref{ex1}).

Note that $X$ is a pseudocompact dense subset of $[0,1]^{\mathfrak{c}}$. It follows that $\beta X=[0,1]^{\mathfrak{c}}$ (see Theorem 1.7.7 in \cite{HTMT}).

By Theorem 2.2 \cite{Os26}, $X$ is a weak $\lambda$-space that is not $\Delta_1$-space. Then $B_1(X,[0,1])$ is cofinally Polish by Theorem \ref{th3} and Corollary \ref{cor10}, and $B_1(X)$ is not cofinally Polish by Theorems \ref{th1} and \ref{th2}.

Note that the countable set $S$ is not $C^*$-embedded in $X$ because if $S$ is $C^*$-embedded then $|\beta S|>|\beta X|=|[0,1]^{\mathfrak{c}}|$. Thus $C_p(X,[0,1])$ is not cofinally Polish by Theorem \ref{th6}.

\end{example}

\begin{example}\label{Ex4} (Example 2 in \cite{Is}) There exists a space $X$ that:

(1) any countable subset of $X$ is closed;

(2) there is a countable subset of $X$ which is not $C^*$-embedded.

Let $A$ be the set of point of $\omega_2+1$ having an uncountable base of neighborhoods, and
$B$ the set of isolated points of $\omega_2+1$. Let us put
$Z=(\omega_2+1)\times (\omega+1)$, $B_n = B\times \{n\}$, $A^*= A\times \{\omega\}$,
$D=\{\omega_2\}\times \mathbb{N}$ and $X=(\bigcup \{B_n: n\in \mathbb{N}\})\bigcup A^*\bigcup D$. Then the
subspace $X$ of $Z$ is a desired space. Indeed, it is routinely
proved that any countable subset is closed, and that $D$ is
not $C^*$-embedded.
\end{example}

Thus Example \ref{Ex4} provides a constructive answer to open questions (Questions 4.3 and 4.4 in \cite{Tkachuk} or Problems 89-90 in \cite{Tk22}). Indeed the space $X$  such that any countable subset of $X$ is closed, but $C_p(X,[0,1])$  does not have a dense cofinally Polish subspace because $D$ is
not $C^*$-embedded (Theorem \ref{th6}).

The following Figure \ref{fig:diagram2} shows the relationship between the study function spaces and examples.

\begin{figure}[H]
\begin{center}
\begin{tikzpicture}[xscale=7.7, yscale=2, >=Stealth]

    \node (TL) at (0,1) {$C_p(X)\text{ is cofinally Polish}$};
    \node (TR) at (1,1) {$B_1(X)\text{ is cofinally Polish}$};
    \node (BL) at (0,0) {$C_p(X, [0,1])\text{ is cofinally Polish}$};
    \node (BR) at (1,0) {$B_1(X, [0,1])\text{ is cofinally Polish}$};

    \draw[line width=0.8pt, double, double distance=2pt, ->] (0.37,1) -- (0.6,1);

    \draw[line width=0.8pt, double, double distance=2pt, ->] (-0.1,0.8) -- (-0.1,0.2);

    \draw[line width=0.8pt, double, double distance=2pt, ->] (0.37,0) -- (0.6,0);

    \draw[line width=0.8pt, double, double distance=2.5pt, ->] (0.9,0.8) -- (0.9,0.2);


     \draw[dashed, ->] (-0.05, 0.2) -- node[pos=0.5] {$\times$}
        node[pos=0.5, right, xshift=-1mm, font=\footnotesize] {\begin{tabular}{c}Ex. 6.1\end{tabular}}
        (-0.05, 0.8);

    \draw[dashed, ->] (0.95, 0.2) -- node[pos=0.5] {$\times$} (0.95, 0.8);
    \node[font=\footnotesize] at (1.05, 0.5) {\begin{tabular}{c}Ex. 6.3\end{tabular}};


    \draw[dashed, ->] (0.8, 1.1) .. controls (0.5, 1.4) .. node[pos=0.5] {$\times$} node[pos=0.5, above, yshift=1mm, font=\footnotesize] {\begin{tabular}{c}Ex. 6.2\end{tabular}} (0.2, 1.1);

    \draw[dashed, ->] (0.8, -0.1) .. controls (0.5, -0.4) .. node[pos=0.5] {$\times$} node[pos=0.5, below, yshift=-1mm, font=\footnotesize] {\begin{tabular}{c}Ex. 6.3\end{tabular}} (0.2, -0.1);



\end{tikzpicture}
\end{center}
\caption{Function relationship between the
study spaces}
\label{fig:diagram2}
\end{figure}
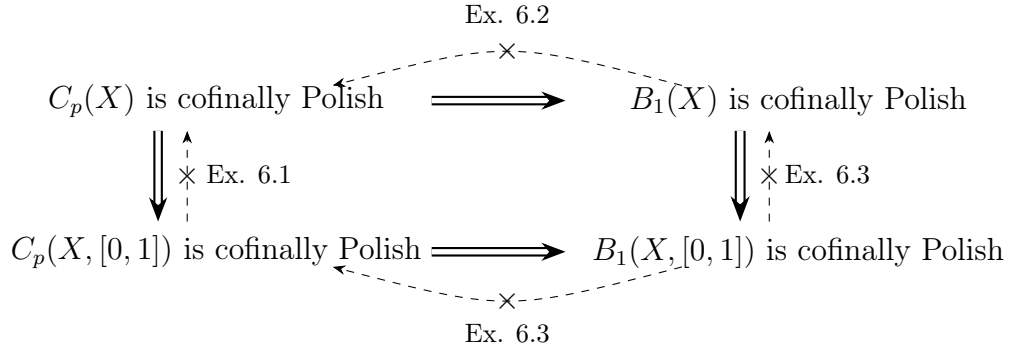

\section{Appendix}

Summarizing all the results, one can formulate-by way of application-theorems containing the largest number (though by no means all) of the equivalences.

\newpage

\begin{theorem} For any space $X$, the following statements are equivalent:

(1) $C_p(X)$ is pseudocomplete;

(2) $C_p(X)$ is a strong Choquet space;

(3) $C_p(X)$ is countably base-compact;

(4) $C_p(X)$ is cofinally Polish;

(5) $C_p(X)$ has a dense cofinally Polish subspace;

(6) $C_p(X)$ is $G_{\delta}$-dense in $\mathbb{R}^X$;

(7) $C_p(X)$ is $\omega$-full in $\mathbb{R}^X$;

(8)  $C_p(X)$ is a Choquet space;

(9) Every countable subset of $X$ is $C$-embedded and closed.

\end{theorem}

\begin{theorem} For any space X, the following statements are equivalent:

(1) $C_p(X,[0,1])$ is pseudocomplete;

(2) $C_p(X,[0,1])$ is pseudocompact;

(3) $C_p(X,[0,1])$ is a strong Choquet space;

(4) $C_p(X,[0,1])$ is countably base-compact;

(5) $C_p(X,[0,1])$ is cofinally Polish;

(6) $C_p(X,[0,1])$ has a dense cofinally Polish subspace;

(7) $C_p(X,[0,1])$ is projectively Menger;

(8) $C_p(X,[0,1])$ is a Choquet space;

(9) $C_p(X,[0,1])$ is $G_{\delta}$-dense in $[0,1]^X$;

(10) $C_p(X,[0,1])$ is $\omega$-full in $[0,1]^X$;

(11) $C_p(X,[0,1])$ is $\sigma$-pseudocompact;

(12) $C_p(X,[0,1])$ is $\sigma$-bounded;

(13) $X$ is $b$-discrete.
\end{theorem}

\begin{theorem} For any space $X$ the following assertions are equivalent:

(1) $B_1(X)$ is pseudocomplete;

(2) $B_1(X)$ is cofinally Polish;

(3) $B_1(X)$ has a dense cofinally Polish subspace;

(4) $B_1(X)$ is a strong Choquet space;

(5) $B_1(X)$ is countably base-compact;

(6) $B_1(X)$ is $G_{\delta}$-dense in $\mathbb{R}^X$;

(7) $B_1(X)$ is $\omega$-full in $\mathbb{R}^X$;

(8)  $B_1(X)$ is a Choquet space;

(9) Every countable subset of $X$ is $B_1$-embedded;

(10) $X$ is a $\Delta_1$-space.

\end{theorem}

\newpage

\begin{theorem} For any space $X$ the following assertions are equivalent:

(1)  $B_1(X,[0,1])$ is pseudocomplete;

(2) $B_1(X,[0,1])$ is cofinally Polish;

(3) $B_1(X,[0,1])$ has a dense cofinally Polish subspace;

(4) $B_1(X,[0,1])$ is pseudocompact;

(5) $B_1(X,[0,1])$ is a strong Choquet space;

(6) $B_1(X,[0,1])$ is countably base-compact;

(7) $B_1(X,[0,1])$ is projectively Menger;

(8) $B_1(X,[0,1])$ is a Choquet space;

(9) $B_1(X,[0,1])$ is $G_{\delta}$-dense in $[0,1]^X$;

(10) $B_1(X,[0,1])$ is $\omega$-full in $[0,1]^X$;

(11) $X$ is a weak $\lambda$-space.

\end{theorem}


\begin{thebibliography}{30}

\bibitem{AL}
J. Aarts, D. Lutzer, Completeness properties designed for recognizing Baire spaces, Diss. Math. 116 (1974) 1-43.





\bibitem{Arh1}
A.V. Arhangel'skii, M.G. Tkachenko, Topological Groups and Related Structures, Atlantis Press/World Scientific,
Amsterdam-Paris, 2008.




\bibitem{Arh2}
A.V. Arhangel'skii, $C_p$-theory, in: K. Kunen, J.E. Vaughan (Eds.), Handbook of Set-Theoretic
Topology, Elsevier Sci. Pub., 1984, pp. 3--56.


\bibitem{BG}
T. Banakh, S. Gabriyelyan, Baire category of some Baire type function
spaces, Topology and its Applications, 272, 2020, 107078.


\bibitem{BH}
T. Banakh and O. Hryniv, Some Baire category properties of topological groups, Visnyk Lviv. Univ. Ser.
Mech.-Mat. 86, 71--76 (2018).

\bibitem{BCM}
 M. Bonanzinga, F. Cammaroto, M. Matveev, Projective versions of selection principles, Topol. Appl. 157 (2010) 874--893.

\bibitem{Eng}
R. Engelking, General Topology, Heldermann, Berlin, 1989.




\bibitem{GJ}
L. Gillman, M. Jerison,  Rings of continuous functions, Van Nostrand,(1960).

\bibitem{HTMT}
M. Hru\v{s}\'{a}k, \'{A}. Tamariz-Mascar\'{u}a and M. Tkachenko (edt.), Pseudocompact Topological groups: A Survey of Results and Problems, Developments in Mathematics, Vol. 55, Springer, Cham, 2019.

\bibitem{Hur}
W. Hurewicz, \"{U}ber eine verallgemeinerung des Borelschen Theorems, Math. Z. 24 (1925) 401--421.


\bibitem{Is}
T. Isiwata, $\kappa$-metric spaces and function spaces, Topology Proceedings, 11, 1986, 291-308.



\bibitem{JSS}
I. Juh\'{a}sz, L. Soukup, Z. Szentmikl\'{o}ssy, $\mathcal{D}$-forced spaces: A new ap-
proach to resolvability, Topology Appl. 156 (2006) 1800--1824.


\bibitem{KKL}
J. K\c{a}kol, O. Kurka, A. Leiderman, Some classes of topological spaces extending the class of $\Delta$-spaces, Proc. Amer. Math. Soc., 152:2 (2024), 883--
898.



\bibitem{Kechris}
A.S. Kechris, Classical Descriptive Set Theory, Springer-Verlag, New York, 1995.



\bibitem{Kur}
K. Kuratowski, Topology, I, PWN, Warszawa, 1966.




\bibitem{NiTkYe}
J. Niknejad, V.V. Tkachuk, L. Yengulalp,  Polish factorizations, cosmic spaces and domain representability.
Bull. Belg. Math. Soc. 25(3), 439--452 (2018).



\bibitem{OS1}
A.V. Osipov, The $\Delta_1$-property of $X$ is equivalent to the Choquet property of $B_1(X)$, Topology and its Applications, 370 (2025), 109395.


\bibitem{Os2}
A.V. Osipov, On Completeness Type Properties of Spaces
of First Lebesgue Functional Class Mappings,  Proceedings of the Steklov Institute of Mathematics
(published in Trudy Instituta Matematiki i Mekhaniki UrO RAN, 2025, Vol. 31, No. 3, 200-214).


\bibitem{Osip1}
A.V. Osipov, Baire property of the space of Baire-one functions. European Journal of Mathematics 11, 8 (2025).

\bibitem{Os26}
A.V. Osipov, Some function applications of weak $\lambda$-spaces, arxiv.org/abs/2608.30278.


\bibitem{Ox}
J.C. Oxtoby, The Banach-Mazur game and Banach category theorem, in Contributions to the Theory of
Games (Princeton Univer. Press, NJ, Princeton, 1957), 159-163.

\bibitem{Pyt}
E.G. Pytkeev, Baire property of spaces of continuous functions. Mathematical Notes of the Academy of Sciences of the USSR 38, 908--915 (1985).


\bibitem{RRS}
V. de O. Rodrigues, V. dos S. Ronchim, P. J. Szeptycki, Special sets of reals and weak form of normality on Isbell-Mr\'{o}wka spaces, Commentationes Mathematicae Universitatis Carolinae, Vol. 64 (2023), No. 1, 109--126.



\bibitem{Sakai}
M. Sakai, The projective Menger property and an embedding of $S_{\omega}$ into
function spaces, Topology and its Applications 220 (2017) 118--130.





\bibitem{Sh}
D.B. Shakhmatov, A pseudocompact Tychonoff space all countable subsets of which are closed and $C^*$-embedded, Topol.
Appl. 22 (1986) 139--144.




\bibitem{Tkachuk}
V.V. Tkachuk, The space $C_p(X)$ is cofinally Polish if and only if it is pseudocomplete. RACSAM 115, 68 (2021).

\bibitem{Tkachuk2}
V.V. Tkachuk, A $C_p$-Theory Problem Book. Topological and Function Spaces. Springer, NewYork (2011).


\bibitem{Tk14}
V.V. Tkachuk,  A $C_p$-Theory Problem Book. Special Features of Function Spaces. Springer, New York
(2014).

\bibitem{Tk86}
V.V. Tka\v{c}uk, The spaces $C_p(X)$: decomposition into a countable union of bounded subspaces and completeness properties,
Topol. Appl. 22 (1986) 241--253.


\bibitem{Tk22}
V.V. Tka\v{c}uk, $C_p$-THEORY IN 2022, Questions and Answers in General Topology, 41 (2023), 75--133.

\end{thebibliography}

\end{document}